\documentclass{article}
\usepackage{amsmath}
\usepackage{amssymb}
\usepackage{amsthm}
\usepackage[margin=1in]{geometry}

\newtheorem{thm}{Theorem}
\newtheorem{lemma}[thm]{Lemma}
\newtheorem{prop}[thm]{Proposition}
\newtheorem{cor}[thm]{Corollary}
\newtheorem{quest}{Question}

\theoremstyle{definition}
\newtheorem{dfn}{Definition}

\theoremstyle{remark}
\newtheorem*{subproof}{Proof}
\newtheorem{claim}{Claim}[thm]

\newcommand{\ZF}{\text{ZF}}
\newcommand{\ZFC}{\text{ZFC}}

\newcommand{\DC}{\text{DC}}
\newcommand{\HOD}{\text{HOD}}

\newcommand{\id}{\text{id}}
\newcommand{\VitEq}{\mathrel{\mathbb{E}_0}} 

\newcommand{\dom}[1]{\text{dom}(#1)}
\newcommand{\nat}{\mathbb{N}}
\newcommand{\Coll}[2]{\text{Coll}({#1}, {#2})}
\newcommand{\card}[1]{\lvert{#1}\rvert}

\newcommand{\colouringPoset}[2]{P_{#1}(#2)}

\newcommand{\powset}[1]{\mathcal P(#1)}

\newcommand{\balanceEq}{\equiv_b}

\title{Ultrafilters, Transversals, and the Hat Game}
\author{Luke Serafin}
\date{\today}

\begin{document}

\maketitle

\begin{abstract}
Geschke, Lubarsky, and Rahn in ``Choice and the Hat Game''~\cite{choice-and-the-hat-game} generalize the classic hat game puzzle to infinitely-many players and ask whether every model of set theory without choice
in which the optimal solution can be carried out contains either a nonprincipal ultrafilter on $\nat$ or else a Vitali set.
A negative answer is obtained here by constructing a model in which there is an optimal solution to the hat game puzzle but no nonprincipal ultrafilter on $\nat$ and no Vitali set.
This is accomplished in a more general setting, establishing that for any Borel bipartite graph $\Gamma$ not embedding some $K_{n,\omega_1}$ and with countable colouring number there is a model of $\ZF + \DC$ in which $\Gamma$ has a
$2$-colouring but there is no ultrafilter as above or Vitali set.
The same conclusion applies to the natural generalization of the hat game to an arbitrary finite number of hat colours.
\end{abstract}

\section{The Hat Game}

Consider a game in which a group of people is arranged in a line such that the person at the back of the line can see everyone else,
the person at the front of the line can see no one else, and all other people see just those who are in front of them.
Everyone closes their eyes and a hat is placed on each player's head.
It is common knowledge that these hats are either black or white, and the goal of the group is to correctly guess as many of their hat colours
as possible.
The rules of the game are as follows:
\begin{itemize}
\item People guess their own hat colours, one at a time, starting at the back of the line and moving to the front.
\item No communication is allowed beyond a single guess of the colour of one's own hat.
\end{itemize}
Assuming the group is allowed to confer and agree on a strategy before receiving their hats, what is the fewest number of wrong guesses
that the group can guarantee?

This is a well-known recreational mathematics puzzle, and as noted in \cite{choice-and-the-hat-game} (but known long before), the
surprising answer is that by having the first player announce the parity of the number of black hats which he or she sees (encoded as a hat colour), the players can guarantee at most one error.
This is clearly optimal.

The hat problem can be naturally generalized in two directions.
The authors of \cite{choice-and-the-hat-game} consider infinite sequences of black and white hats and show that players can still achieve
at most one error by means of a \emph{parity function}, which is a $\{0, 1\}$-valued function $p$ on the collection of all infinite sequences
of zeroes and ones with the property that if two such sequences $x$ and $y$ differ by a single bit, then $p(x) \ne p(y)$.
In fact, the full power of the axiom of choice is not needed, and the authors observe that either a nonprincipal ultrafilter on $\nat$ or an
$\VitEq$-transversal (otherwise known as a Vitali set) suffices.
They then ask whether the existence of a parity function implies the existence of either of these objects (\cite{choice-and-the-hat-game},
Question 12).
We answer this question with the following:
\begin{thm}
Assuming the consistency of the existence of an inaccessible cardinal, there is a model of set theory in which there is a parity function
but there is neither a nonprincipal ultrafilter on $\nat$ nor an $\VitEq$-transversal.
\end{thm}

The other direction in which the hat game can be generalized is with hats of more than one colour.
For finitely-many colours (and players), essentially the same strategy as for two colour works to ensure at most one error: Players agree
on a fixed enumeration $0, 1, \ldots, k - 1$ of the possible colours and the first one states the sum, modulo the number of colours, of the colours
of the hats which he or she sees.
This can be generalized to infinitely-many people as in the case of two colours---leading to the notion of a \emph{$k$-arity function}, and
again the existence of either a nonprincipal ultrafilter on $\nat$ or of an $\VitEq$-transversal suffices to construct one of these functions.
We prove also that the existence of a $k$-arity function does not imply the existence of either a nonprincipal ultrafilter on $\nat$ or of an
$\VitEq$-transversal, by a similar argument as in the case of two colours.

Parity functions can be considered as $2$-colourings of a certain graph, the \emph{Hamming graph}, which has all infinite sequences of zeroes and ones as its vertices and which
connects two vertices by an edge precisely when they differ by a single bit.
Because it requires little extra effort, we actually prove a much more general result than what was stated above, and show that for
any sufficiently-nice bipartite graph $\Gamma$ on the set of infinite sequences of zeroes and ones (and again assuming the consistency
of an inaccessible cardinal), there is a model of set theory in which there is a $2$-colouring of $\Gamma$ but there is neither a nonprincipal
ultrafilter on $\nat$ nor an $\VitEq$-transversal.

\section{Preliminaries}

Our terminology and notation are mostly standard and can be found in references such as \cite{kunen}, \cite{jech}, and \cite{kechris-classical}.
One point of clarification is that for us, a \emph{poset} is a preordered set (i.e. a set together with a reflexive and transitive relation), not a partially ordered set.
We make extensive use of techniques of so-called \emph{geometric set theory}, for which see~\cite{larson-zapletal-GST}, and mostly follow the terminology and notational conventions established in that book.

We shall use the following for technical reasons.
\begin{thm}[Lusin-Novikov (for a proof see {\cite[18.10]{kechris-classical}})] \label{lusin-novikov}
Let $R$ be a Borel relation between standard Borel spaces $X$ and $Y$.
If $R$ has countable vertical sections, then there is a countable family $\{ f_n \mathrel{:} n \in \nat \}$ of Borel partial functions from $X$ to $Y$ such that $R = \bigcup_{n \in \nat} f_n$.
\end{thm}

For us the \emph{axiom of dependent choices}, abbreviated $\DC$, states that for any relation $R$ on a set $X$ satisfying the hypothesis that for every $x \in X$ there is $y \in X$ with $x \mathrel{R} y$, there is a countable sequence
$\langle x_n \mathrel{:} n \in \nat \rangle$ such that for each $n$, $x_n \mathrel{R} x_{n+1}$.
By a model of set theory we mean a model of $\ZF + \DC$.
In particular, we do not assume that the full axiom of choice holds in our models of set theory, though we do insist that $\DC$ holds.
This weaker form of choice suffices for much of analysis and prevents pathologies such as the set of real numbers being a countable union of countable sets (for which possibility see \cite[Theorem 10.6]{jech-ac}).
Two references among many for the axioms of $\ZF$ are \cite{kunen} and \cite{jech}.
We assume the full axiom of choice in the metatheory.

\subsection{The Symmetric Solovay Model $W$}

The \emph{symmetric Solovay model}, denoted $W$, is a natural starting point for independence results over models of $\ZF + \DC$.
It is in particular a model of $\DC$, and was constructed by Solovay~\cite{solovay} to demonstrate the possibility, upon dropping the axiom of choice, of all sets of reals being Lebesgue measurable.
In particular, in $W$ there is no nonprincipal ultrafilter on $\nat$ and no $\VitEq$ transversal, as no example of either of these objects can be Lebesgue measurable.

From this point forward let $\lambda$ be a fixed inaccessible cardinal.
Fix also a $V$-generic filter $\mathfrak G$ for the Lévy-collapse forcing $\Coll{\aleph_0}{<\lambda}$, which collapses all cardinals below $\lambda$ to be countable and collapses $\lambda$ to be $\aleph_1$.
Though it was not Solovay's original definition (for which see \cite{solovay}), following \cite{larson-zapletal-GST} we define $W = \HOD(V \cup 2^\nat)^{V[\mathfrak G]}$.
Note that here the set $2^\nat$ is computed in $V[\mathfrak G]$, not in $V$.

\subsection{Elementary Forcing}

Because we are interested in models of $\ZF + \DC$, we want any forcing extensions of these models to preserve $\DC$.
For us this is accomplished as follows.

\begin{dfn}
A poset $P$ is \textit{$\sigma$-closed} iff for every infinite descending chain $p_0 \ge p_1 \ge p_2 \ge \cdots$ in $P$ there is a condition $p$ with $p \le p_i$ for every $i$.
\end{dfn}

\begin{prop}[\emph{see, e.g. {\cite{aspero-karagila}}}]
If $P$ is a $\sigma$-closed poset, $M \models \DC$, and $G$ is $M$-generic for $P$, then $M[G] \models \DC$.
\end{prop}

A poset $\langle P, \le \rangle$ is \emph{Suslin} iff there is a Polish space $X$ over which $P$, $\le$, and $\perp$ are all analytic.
This definability requirement allows one to speak coherently about what it means to be a condition of a Suslin forcing in a generic extension.

We say that a property $\Phi$ of a Suslin forcing $P$ holds \emph{cofinally} iff for every generic extension $V[G]$ there is a further generic extension $V[G][H]$ in which $P$ has property $\Phi$.

\subsection{Balanced, Bernstein-Balanced, and Compactly-Balanced Forcing}

We now turn to the technical apparatus of our proofs, which is a body of techniques known as \emph{geometric set theory} \cite{larson-zapletal-GST}.
In geometric set theory one works with forcing notions which satisfy desirable properties with respect to not just conditions in the model at hand, but also objects which become conditions in certain generic extensions.
These objects are known as \emph{virtual conditions}, and we shall need to establish some preliminary definitions before defining some of the core classes of forcing notions studied in geometric set theory.

Virtual conditions are, intuitively, objects which exist in $V$ and describe potential conditions of a Suslin forcing which are guaranteed to be consistent across forcing extensions where they are realized.
More concretely, they determined by pairs\footnote{Specifically, they are equivalence classes of such pairs under a relation of \emph{virtual equivalence}, for which see \cite{larson-zapletal-GST}} $\langle Q, \dot p \rangle$
such that $Q$ forces that $\dot p$ is a condition in $P$ and $Q \times Q \Vdash \dot p / \dot G_0 \equiv \dot p / \dot G_1$, where $\dot G_1$ and $\dot G_2$ are canonical names for the projection of the generic filter name $\dot G$ onto
the first and second coordinates, respectively; $\equiv$ is the equivalence relation induced by the preorder $P$, and for $G$ a generic filter and $\dot q$ a name $\dot q / G$ is the interpretation of $\dot q$ in the generic extension
$V[G]$.
A virtual condition $\langle Q, \dot p \rangle$ is \emph{realized} precisely in those forcing extensions whose posets are forcing equivalent to posets of the form $Q \times R$ for some $R$.
See \cite{larson-zapletal-GST} for the detailed definition and basic properties of virtual conditions.

\begin{dfn}[{\cite[Definitions 5.2.1,10]{larson-zapletal-GST}}] 
Let $P$ be a Suslin forcing.
A virtual condition $\overline p$ of $P$ is \emph{balanced} iff for all pairs $V[G]$, $V[H]$ of mutually-generic extensions of $V$ and all conditions $p \in V[G]$, $q \in V[H]$ with $p, q \le \overline p$ in some generic extension
$V[K]$ realizing $\overline p$ and containing the generic filters $G$ and $H$, $p$ and $q$ are compatible in $V[K]$.
Working in $V$ again, we say the poset $P$ is \emph{balanced} iff for every condition $p \in P$ there is a balanced virtual condition $\overline p \le p$.
\end{dfn}

The fact that a notion of forcing is balanced generally finds use in proofs via contradiction.
One assumes that some object, such as a nonprincipal ultrafilter on $\nat$ or an $\VitEq$-transversal, exists in a balanced extension
of $W$, and then finds a pair of conditions below a balanced virtual condition which are incompatible, thereby contradicting that the extension was balanced.
Of fundamental importance to establishing such contradictions is the following.

\begin{prop}[{\cite[Proposition 5.2.4]{larson-zapletal-GST}}]
If $\langle Q, \dot p \rangle$ determines a balanced virtual condition, $\lvert Q \rvert < \lambda$, $\dot G$ is the canonical $P$-name for a $P$-generic filter, and $\phi(x)$ is a formula with paremeters in $V$,
then $\dot p$ determines the truth or falsity of $\phi(\dot G)$, in the sense that one of the following holds:
\begin{itemize}
\item $Q \Vdash \Coll{\aleph_0}{< \lambda} \Vdash \dot p \Vdash_P W[\dot G] \models \phi(\dot G)$, \textrm{or}
\item $Q \Vdash \Coll{\aleph_0}{< \lambda} \Vdash \dot p \Vdash_P W[\dot G] \models \neg \phi(\dot G)$.
\end{itemize}
\end{prop}

In order to demonstrate that a forcing is balanced it is often helpful to classify balanced pairs, and for that purpose the following
equivalence relation is useful, as it provides a means of reducing balanced pairs to balanced virtual conditions.

\begin{dfn}[{\cite[Definition 5.2.5]{larson-zapletal-GST}}]
Pairs $\langle Q, \dot q \rangle$, $\langle R, \dot r \rangle$ are \emph{balance-equivalent}, denoted
$\langle Q, \dot q \rangle \balanceEq \langle R, \dot r \rangle$ iff for all pairs
$\langle Q', \dot q' \rangle \le \langle Q, \dot q \rangle$, $\langle R', \dot r' \rangle \le \langle R, \dot r \rangle$,
\[ Q' \times R' \Vdash \exists q \in \dot q' \ \exists r \in \dot r' \ \exists p. \ p \le q, r. \]
\end{dfn}

That $\balanceEq$ is indeed an equivalence relation is established in \cite[prop. 5.2.6]{larson-zapletal-GST},
which also proves that if $\langle Q, \dot q \rangle \le \langle R, \dot r \rangle$, then
$\langle Q, \dot q \rangle \balanceEq \langle R, \dot r \rangle$.

An important property of balance equivalence is that every balance equivalence class includes a virtual condition, which is in fact
unique up to equivalence of virtual conditions, so when working with pairs up to balance equivalence it suffices to consider
virtual conditions.

\begin{prop}[{\cite[thm. 5.2.8]{larson-zapletal-GST}}]
For any Suslin forcing $P$, every balance equivalence class of pairs includes a virtual condition which is unique up to
equivalence of virtual conditions.
\end{prop}

\subsection{Bernstein Balanced Forcing}

\begin{dfn}
Let $P$ be a Suslin forcing.
A virtual condition $\overline p$ is \emph{Bernstein-balanced} iff for every generic extension $V[G]$, every $p \le \overline p$ in
$V[G]$, every poset $Q \in V$ such that $\powset{Q} \cap V$ is countable, and for every perfect collection $\mathcal H$ of filters
mutually generic over $Q$ there is a filter $H \in \mathcal H$ such that all elements of $P$ in $V[H]$ below $\overline p$ are compatible with $p$.
\end{dfn}

The definition of a Bernstein-balanced forcing is now as expected.

\begin{dfn}
A Suslin forcing $P$ is \emph{Bernstein-balanced} iff for every $p \in P$ there is a Bernstein-balanced virtual condition
$\overline p \le p$.
\end{dfn}

The following theorem is for us the main use of the notion of Bernstein balance.

\begin{thm}[{\cite[Theorem 12.2.3]{larson-zapletal-GST}}] \label{bernstein-no-ultra}
In cofinally Bernstein-balanced extensions of the symmetric Solovay model, there is no nonprincipal ultrafilter\footnote{In fact, more is true: there is no diffuse finitely-additive probability measure on $\nat$.} on $\nat$.
\end{thm}

\subsection{Compactly-balanced Forcing}

A strengthening of the notion of balanced forcing guarantees that the Harrington-Kechris-Louveau dichotomy is reflected in the cardinalities of the quotient spaces of Borel equivalence relations.
The authors of~\cite{larson-zapletal-GST} call this \emph{preservation of the smooth divide}\footnote{Equivalence relations Borel reducible to $\id$ are often called ``smooth.''}.

\begin{dfn}[\cite{larson-zapletal-GST}]
A Suslin forcing $P$ is \emph{compactly balanced} iff it is balanced and its balanced virtual conditions are classified by a definable compact Hausdorff space $\langle B, \mathcal T \rangle$ such that downward cones of conditions
in $P$ are nonempty and closed, any balanced virtual condition realized in a generic extension $V[H]$ can be extended in any larger generic extension $V[G] \supseteq V[H]$, and the $\le$-relation between virtual conditions which are
elements of a generic extension $V[H]$ and those in a larger generic extension $V[G]$ is closed in $\mathcal T^{V[H]} \times \mathcal T^{V[G]}$ (from the perspective of $V[G]$).
\end{dfn}

For further details on the definition of compactly-balanced forcing, see \cite[Definition 9.2.1]{larson-zapletal-GST}.
The authors are intentionally noncommital about the precise meaning of the phrase ``definable Hausdorff space,'' though they do note that real parameters are permitted.
If one views $V$ as a countable transitive model of ZFC which exists within an ambient universe $\mathcal V$, then definability in $V$ with real parameters from $\mathcal V$ suffices.
The required definitions in this paper are clearly allowed, in any case.

\begin{thm}[{\cite[Theorem 9.2.2]{larson-zapletal-GST}}] \label{compact-balance-smooth-divide}
Let $P$ be a compactly-balanced forcing and $G$ be a $P$-generic filter over the symmetric Solovay model $W$.
Then $W[G]\models \card{2^\nat / \VitEq} > 2^{\aleph_0}$.
\end{thm}

Note that $\card{2^\nat / \VitEq} > 2^{\aleph_0}$ implies that there is no $\VitEq$-transversal, because if there were then it would determine an injection from $2^\nat$ into $2^\nat / \VitEq$.

\subsection{Colouring Graphs}

We now fix some notation and terminology for the graph-theoretic aspects of our work.

\begin{dfn}
Let $\Gamma$ be a graph on a set $X$.
A \emph{colouring} of $\Gamma$ is a function $c : X \rightarrow \kappa$ where $\kappa$ is an initial ordinal (so a wellordered cardinal) and $c$ has the property that if $x \mathrel \Gamma y$, then $c(x) \ne c(y)$.
\end{dfn}

The cardinal $\kappa$ can be thought of as the set of colours, in which case a colouring assigns colours in such a way that adjacent vertices receive distinct colours.

One is often concerned with colourings which use the least possible number of colours.

\begin{dfn}
Let $\Gamma$ be a graph\footnote{An irreflexive, symmetric relation.} on a set $X$.
The \emph{chromatic number} $\chi(\Gamma)$ of $\Gamma$ is the least ordinal $\kappa$ such that there is a colouring $c : X \rightarrow \kappa$ of $\Gamma$.
\end{dfn}

Following~\cite{larson-zapletal-GST}, for purposes of adding optimal colourings in balanced extensions of the symmetric Solovay model to graphs with countable chromatic number, we need a strengthening of the notion of countable
chromatic number.

\begin{dfn}[{\cite{erdos-hajnal-chnum-gr-ss}}]
A graph $\Gamma$ on a set $X$ has \emph{countable colouring number} iff there is a wellorder $\prec$ on $X$ such that for each $y \in X$ there are only finitely-many $x \prec y$ such that $x \mathrel{\Gamma} y$.
\end{dfn}

In $V$ graphs with countable colouring number have countable chromatic number, since one can wellorder the elements of $X$ and then construct a colouring with at most $\aleph_0$ colours by transfinite recursion, choosing for each
$y \in X$ a colour not included in the finite set $\{ c(x) : x \prec y \mathrel{\&} x \mathrel{\Gamma} y \}$.

\section{Parity Functions}

At last we arrive at the main theorem of this work.

\begin{thm} \label{main-thm}
Assuming the existence of an inaccessible cardinal, there is a transitive model $M$ of $\ZF + \DC$
which contains a parity function for $2^\nat$ but neither a nonprincipal ultrafilter on $\nat$ nor an $\VitEq$-transversal.
\end{thm}

The proof is accomplished by considering the balanced extension of $W$ by a forcing notion adding a $2$-colouring of the Hamming graph $H$ on $\nat$.
This can be accomplished by means of a simplicial complex forcing (see \cite[Example 6.2.11]{larson-zapletal-GST}), but we illustrate the use of a different Suslin forcing, modified from \cite[Definition 8.1.1]{larson-zapletal-GST},
which has the advantage of being balanced for all Borel bipartite graphs of countable colouring number rather than just those which are locally countable.

\begin{dfn} \label{2-colouring-poset-df}
For $X$ a Polish space and $\Gamma$ a graph on $X$, the \emph{$2$-colouring poset $\colouringPoset{2}{\Gamma}$} consists of countable partial $2$-colourings of $\Gamma$ which can be extended to total colourings, ordered by reverse
inclusion.
\end{dfn}

We shall use the following technical lemma.
Here and in the sequel, a ``large finite fragment of $\ZFC$'' is a collection of axioms of $\ZFC$ containing all axioms which have length less than $16^{256}$ symbols.
What is important is that the fragment contains all axioms of $\ZFC$ relevant to the argument at hand, and the above requirement will always be clearly sufficient.
For any fixed finite fragment of $\ZFC$, a transitive set model of this fragment can be proved to exist via reflection (see, for example, \cite[thm. 7.5]{kunen}).

\begin{lemma}[{\cite[claim 8.1.3]{larson-zapletal-GST}}] \label{ctm-finite-gr-nbhd}
Fix a Polish space $X$ and $\Gamma$ a Borel graph on $X$ with countable colouring number, as witnessed by a wellorder $\prec$.
Let $N$ be a transitive model of a large finite fragment of $\ZFC$ which contains $X$, $\Gamma$, and $\prec$, and $M$ be a countable elementary submodel of the structure $\langle N, \in, X, \Gamma, \prec \rangle$.
Then any vertex $y \in X \setminus M$ has only finitely-many neighbors in $M$.
\end{lemma}

To verify DC is preserved we need to pass to a dense subset of $\colouringPoset{2}{\Gamma}$, as in the development in \cite[sec. 8.1]{larson-zapletal-GST}.

\begin{lemma}
For every bipartite Borel graph $\Gamma$ over a Polish space $X$ which has countable colouring number, $\colouringPoset{2}{\Gamma}$ contains a dense subset $P$ which is Suslin and $\sigma$-closed.
\end{lemma}

\begin{proof}
Analogously as for \cite[definition 8.1.1]{larson-zapletal-GST},
let $P$ be the set of enumerations of partial functions $p$ from $X$ to $2$ with the property that every $x \in \dom{p}$ is $\Gamma$-connected with at most finitely-many elements of $X \setminus \dom{p}$, and moreover if
$x, y \in \dom{p}$ and $p(x) = p(y)$, then $d_\Gamma(x,y) \in 2 \nat \cup \{ \infty \}$, where $d_\Gamma$ is the graph metric with edge weights $1$.
The latter condition is needed to ensure that $p$ can be extended to a total $2$-colouring of $\Gamma$.

\begin{claim}
$P \subseteq \colouringPoset{2}{\Gamma}$
\end{claim}

\begin{subproof}
Let $p \in P$ and $c : X \rightarrow 2$ be a total $2$-colouring of $\Gamma$.
Define $d : X \rightarrow 2$ by
\[ d(x) = \begin{cases} p(x) & \text{if $x \in \dom{p}$} \\
                        p(y) & \text{if there is $y \in \dom{p}$ with $d_\Gamma(x,y) \in 2 \nat$} \\
                    1 - p(y) & \text{if there is $y \in \dom{p}$ with $d_\Gamma(x,y) \in 2 \nat + 1$} \\
                    c(x)     & \text{otherwise.} \end{cases} \]
This function $d$ is well-defined because $p \in P$ and $\Gamma$ is bipartite, and it is clearly a total $2$-colouring of $\Gamma$ extending $p$.
\end{subproof}

\begin{claim}
$P$ is dense in $\colouringPoset{2}{\Gamma}$.
\end{claim}

\begin{subproof}
Let $p \in \colouringPoset{2}{\Gamma}$, and $c : X \rightarrow 2$ be a total $2$-colouring of $\Gamma$ which extends $p$.
Fix a wellorder $\prec$ witnessing that $\Gamma$ has countable colouring number.
Let $\overline M$ be a model of a large finite fragment of $\ZFC$ which contains $p$, $c$, and $\prec$; and let $M$ be a countable elementary substructure of $(M, \in, p, c, \prec)$.
By lemma~\ref{ctm-finite-gr-nbhd}, $\{ y \in X : \exists^\infty x \in M. \ x \mathrel{\Gamma} y \} \subseteq M$.
Thus $c \restriction M$ is a condition in $P$ extending $p$.
\end{subproof}

\begin{claim}
The set of enumerations of conditions in $P$ is Borel.
\end{claim}

\begin{subproof}
Similar to \cite[Claim 8.1.4]{larson-zapletal-GST}.
By lemma~\ref{ctm-finite-gr-nbhd}, for every countable set $A \subseteq X$ the set of $y \in X$ adjacent to infinitely-many elements of $A$ is countable.
Using the Lusin-Novikov theorem (theorem~\ref{lusin-novikov}), choose a family of Borel functions $f_k \mathrel{:} X^\nat \rightarrow X$ such that for each $a \in X^\nat$, thought of as an enumeration of a countable subset of $X$,
all elements of $X$ which are adjacent to infinitely-many elements of the range of $a$ are in the set $\{f_k(x) \mathrel{:} k \in \nat \}$ of values of the functions $f_k$ evaluated at $x$.

To see that the property of being a condition in $P$ is Borel, note that a function $r \mathrel{:} \nat \rightarrow X \times 2$ enumerates a condition in $P$ iff the range of $r$ is a $\Gamma$-colouring, there are no pairs
$x, y \in X$ with the property that $p(x) = p(y)$ and $d_\Gamma(x,y) \in 2 \nat + 1$, and the range $a$ of $\pi_1 \circ r$ contains all points $f_k(\pi_1 \circ r)$ for $k \in \nat$ which are adjacent to infinitely many elements of $a$,
where $\pi_1$ is the projection map of $X \times \nat$ to the first coordinate. This clearly yields a Borel definition of membership in $P$.
\end{subproof}

It follows that $P$ is Suslin, because two conditions in $P$ are compatible iff they agree on their domains and do not assign the same colour to any two vertices an odd distance apart, which is a Borel property.

\begin{claim}
$P$ is $\sigma$-closed.
\end{claim}

\begin{subproof}
Let $\langle p_i \mathrel{:} i \in \nat \rangle$ be a descending sequence of conditions in $P$.
Define $\chi = \bigcup_{i = 1}^\infty p_i$ and note that this is a function.
If the distance between vertices $x$ and $y$ in $\Gamma$ is odd then no condition $p_i$ colours both with the same colour, so $p$ cannot colour $x$ and $y$ with the same colour either.
We now need to extend $\chi$ to an element of $P$.
Let $N$ be a transitive model of a large fragment of $\ZFC$ such that the elements of $N$ include $X$, $\Gamma$, $\chi$, $\prec$, and $\langle p_i \mathrel{:} i \in \nat \rangle$.
Take $M$ to be a countable elementary substructure of $\langle N, \in, X, \Gamma, \chi, \prec, \langle p_i \mathrel{:} i \in \nat \rangle \rangle$.
Extend $\chi$ to a $2$-colouring $p$ of $\Gamma$ restricted to $M \cap X$; we verify that this is an element of $P$.
The partial $2$-colouring $p$ is countable in $V$ because $M$ is countable.
The fact that $p$ is a partial $2$-colouring of $M \cap X$ immediately implies that if $p(x) = p(y)$, then $x$ and $y$ are separated by an even distance in $\Gamma$.
Since $\prec \cap M$ witnesses that $\Gamma$ has countable colouring number in $M$, each vertex in $\dom{p}$ is adjacent to only finitely-many elements of $ (X \cap M) \setminus \dom{p}$.
From lemma~\ref{ctm-finite-gr-nbhd} we conclude that each vertex in $\dom{p}$ is adjacent to only finitely-many elements of $X \setminus \dom{p}$, and this completes the verification that $p \in P$.
Clearly $p_i \subseteq p$, and thus $p \le p_i$, for each $i \in \nat$.
\end{subproof}
\end{proof}

\begin{lemma} \label{bvc-classification-2-colouring-poset}
The balanced virtual conditions of $P$ are classified by total $2$-colourings of $P$ in the ground model.
\end{lemma}

\begin{proof}
The proof proceeds like that of \cite[thm. 8.1.2]{larson-zapletal-GST}.
Let $R$, $S$ be posets and $\dot x$, $\dot y$ be $R$- and $S$-names for elements of $X \setminus V$, respectively.

\begin{claim} \label{R-forces-x-gamma-fin}
The poset $R$ forces that $\dot x$ is adjacent in $\Gamma$ to at most finitely-many elements in $V$.
\end{claim}

\begin{subproof}
Suppose for contradiction that there is a condition $r \in R$ which forces that $\dot x$ is adjacent in $\Gamma$ to infinitely-many elements of $V$.
Let $N$ be a transitive model of a large finite fragment of $\ZFC$ whose elements include $X$, $\Gamma$, $R$, and $\dot x$.
Take $M$ a countable elementary substructure of $\langle N, \in, X, \Gamma, R, \dot x \rangle$, and $g$ an $M$-generic filter over $R \cap M$.
Denoting the interpretation of $\dot x$ by $x$ (which is an element of $X$), we have that ``$x$ is adjacent in $\Gamma$ to infinitely-many elements of $V$'' holds in $M[g]$ by the forcing theorem.
By Mostowski absoluteness \cite[thm. 25.4]{jech}
$x$ is adjacent in $\Gamma$ to infinitely-many elements of $M$, which contradicts lemma \ref{ctm-finite-gr-nbhd}.
\end{subproof}

\begin{claim} \label{mgen-not-adj}
The poset $R \times S$ forces that $\dot x$ and $\dot y$ are not adjacent in $\Gamma$.
\end{claim}

\begin{subproof}
Suppose for contradiction that there is a condition $\langle r, s \rangle \in R \times S$ which forces that $\dot x$ and $\dot y$ are adjacent in $\Gamma$.
Let $N_1$ be a transitive model of a large finite fragment of $\ZFC$ whose elements contain $X$, $\Gamma$, $R$, $S$, $r$, $s$ $\dot x$, and $\dot y$, and take $N_2$ a transitive model of a large finite fragment of $\ZFC$ whose elements include
$N_1$ and such that for every ordinal $\alpha \in N_1$, $2^{\lvert \alpha \rvert} \in N_2$.
Define $M$ to be the transitive collapse of a countable elementary substructure of $\langle N_1, \in, X, \Gamma, R, S, r, s, \dot x, \dot y \rangle$, and $N$ to be the transitive collapse of a countable elementary substructure of
$\langle N_2, \in, X, \Gamma, R, S, r, s, \dot x, \dot y \rangle$ with $M \in N$.
Let $\langle g_i \mathrel{:} i \in \nat \rangle$ be filters in $N$ which are mutually $M$-generic over $R \cap M$.
By mutual-genericity in $M$, the interpretations $x_i$ of $\dot x$ with respect to each $g_i$ are distinct elements of $X \cap N$.
Now take a filter $h$ containing $s$ which is $N$-generic over $S \cap M$, and let $y = \dot y / h$.
Note that $y \in X \setminus N$.
The forcing theorem yields that each $x_i$ is adjacent to $y$ in $M[g_i, h]$, and by Mostowski absoluteness this holds in $V$.
But that contradicts lemma \ref{ctm-finite-gr-nbhd}.
\end{subproof}

\begin{claim}
For every total colouring $c \mathrel{:} X \rightarrow 2$ of $\Gamma$ the pair $\langle \Coll{\aleph_0}{X}, \check c \rangle$ is balanced.
\end{claim}

\begin{subproof}
Fix a total $2$-colouring $c$ of $\Gamma$ and observe that claim~\ref{R-forces-x-gamma-fin} entails that $\Coll{\aleph_0}{X} \Vdash \check c \in P$.
Let $\dot \chi$, $\dot \rho$ be $R$- and $S$-names, respectively, for elements of $P$ extending $\check c$.
Work now in a generic extension of $V$ by $R \times S$, and let $\chi$ and $\rho$ be the interpretations of $\dot \chi$ and $\dot \rho$ by the projections of the generic filter to its first and second factors, respectively.
By the product forcing theorem the intersection of the domains of $\chi$ and $\rho$ is a subset of $V$, and moreover $\chi \restriction (X \cap V) = \rho \restriction (X \cap V) = \check c$.
Therefore $\chi \cup \rho$ is a function.
Claim \ref{mgen-not-adj} immediately implies that it is also a $2$-colouring, and consequently $R \times S$ forces that $\dot \chi$ and $\dot \rho$ are compatible, as required in the definition of a balanced pair.
\end{subproof}

\begin{claim}
Every balanced pair for $P$ is balance-equivalent to a pair of the form $\langle \Coll{\aleph_0}{X}, \check c \rangle$ where $c \mathrel{:} X \rightarrow 2$ is a $\Gamma$-colouring.
\end{claim}

\begin{subproof}
Let $\langle Q, \dot q \rangle$ be a balanced pair.
Because every balanced pair is balance-equivalent to a virtual condition \cite[thm. 5.2.8]{larson-zapletal-GST},
we may assume that $\langle Q, \dot q \rangle$ is also a virtual condition.
Using the fact that a strengthening of a balanced pair is also balanced \cite[prop. 5.2.6]{larson-zapletal-GST},
we may assume without loss of generality that $Q \Vdash X \cap V \subseteq \dom{\dot q}$.
Because $\langle Q, \dot q \rangle$ is balanced, for $x \in X$ there is $c_x < 2$ such that $Q \Vdash \dot q(x) = c_x$.
Let $c$ be the mapping $x \mapsto c_x$, which exists and is an element of $V$ by the axiom of choice applied in $V$.
Since $Q \Vdash \dot q \in P$, $c$ is a $2$-colouring of $\Gamma$ in $V$.
Clearly $Q \times \Coll{\aleph_0}{X} \Vdash \dot q \le \check c$, and so by \cite[prop. 5.2.6]{larson-zapletal-GST} and the preceding claim we have that $\langle Q, \dot q \rangle$ is balance-equivalent
to $\langle \Coll{\aleph_0}{X}, \check c \rangle$.
\end{subproof}

\begin{claim}
If $c \ne d$ are distinct $\Gamma$-colourings then $\langle \Coll{\aleph_0}{X}, \check c \rangle$ is not balance-equivalent to $\langle \Coll{\aleph_0}{X}, \check d \rangle$.
\end{claim}

\begin{subproof}
This follows directly from the definition of the order of $P$.
\end{subproof}
\end{proof}

\begin{lemma}
\label{balance-of-2-colouring-poset}
For $\Gamma$ a bipartite Borel graph over a Polish space $X$ which has countable colouring number, there is a balanced Suslin forcing $P$ which adds a $2$-colouring of $\Gamma$ to the symmetric
Solovay model $W$.
\end{lemma}

\begin{proof}
It is immediate from the preceding lemma \ref{bvc-classification-2-colouring-poset} that $P$ is balanced.
Since $P$ is dense in $\colouringPoset{2}{\Gamma}$, forcing with $P$ adds a $2$-colouring of $\Gamma$ to $W$.
\end{proof}

We use Bernstein-balance together with another technical assumption about $\Gamma$ to establish that $W^P$ contains no nonprincipal ultrafilter on $\nat$.

\begin{lemma}
Assume that $\Gamma$ is a Borel graph on a Polish space $X$ with countable colouring number, and that there is $n \in \nat$ such that there is no embedding\footnote{That is, injective graph homomorphism.} from $K_{n, \omega_1}$ into
$\Gamma$.
The poset $P$ which adds a $2$-colouring of $\Gamma$ is Bernstein-balanced and every balanced virtual condition is also Bernstein-balanced.
\end{lemma}

\begin{proof}
The proof follows closely the argument from Example 12.2.11 in \cite{larson-zapletal-GST}.

By lemma~\ref{bvc-classification-2-colouring-poset} an arbitrary balanced virtual condition for $P$ is balance-equivalent to a condition of the form $\langle \Coll{\aleph_0}{X}, \check c \rangle$, where $c \mathrel{:} X \rightarrow 2$
is a $2$-colouring of $\Gamma$.
Work now in a generic extension $V[G]$.
Proceeding by contradiction, assume that $\langle \Coll{\aleph_0}{X}, \check c \rangle$ is not Bernstein-balanced.
Then there is a condition $p \le \check c$, an infinite poset $Q \in V$ with $\lvert \mathcal P(Q) \cap V \rvert^{V[G]} = \aleph_0$, and an uncountable (in $V[G]$) family $\mathcal H$ of mutually $V$-generic filters over $Q$
such that for each $H \in \mathcal H$ there is a condition $p_H$ in $V[H]$ with $p_H \le \check c$ and $p_H \perp p$.
For $x \in \dom{p}$ let $\mathcal G_x$ be the subfamily of $\mathcal H$ consisting of filters $H$ such that there is $x_H \in X \cap (V[H] \setminus V)$ with $x_H \mathrel{\Gamma} x$ and $p_H(x_H) = p(x)$.
The fact that $p_H \perp p$ for each $H \in \mathcal H$ guarantees that $\mathcal G_x$ is nonempty for each $x \in \dom{p} \setminus V$, and that $\bigcup_{x \in \dom{p}} \mathcal G_x = \mathcal H$.
For each $H \in \mathcal H$ and $x_H \in X \cap V[H]$ we have by lemma~\ref{ctm-finite-gr-nbhd} that $x_H$ is adjacent to only finitely-many $x \in X \setminus V[H]$, and so by the pigeonhole principle applied with the fact that
$\mathcal H$ is uncountable, there is $x \in \dom{p}$ such that $\mathcal G_x$ is uncountable.
Fix such $x$ and some particular choice of $x_H$ as above for each $H \in \mathcal G_x$.
Fix $n$ such that $K_{n, \omega_1}$ does not embed in $\Gamma$.
Let $a \in [\mathcal G_x]^{n+1}$, and for each $b \in [a]^n$ define $C_b$ to be the set of those $y \in X$ which are adjacent to $x_H$ for every $H \in b$.
Since $K_{n, \omega_1}$ does not embed in $\Gamma$, $C_b$ is countable for each $b$.
Shoenfield absoluteness \cite{shoenfield-absoluteness} (see \cite{jech} for a modern statement and proof) 
entails that $C_b \subseteq V[H \mathrel{:} H \in b]$, and so because each $H \in b$ is an element of $\mathcal G_x$ it follows that $x \in V[H \mathrel{:} H \in b]$.
Since this holds for every $b \in [a]^n$, by mutual genericity of the filters in $a$ we find that $x \in V$.
But then $p(x) = \check c(x) = p_H(x)$, contradicting the fact that $p_H$ is a partial $2$-colouring of $\Gamma$ (because $x \mathrel{\Gamma} x_H$ and $p_H(x_H) = p(x)$).
\end{proof}

\begin{cor}
The model $W^P$ for $\Gamma$ as in the lemma does not contain a nonprincipal ultrafilter on $\nat$.
\end{cor}

\begin{proof}
Immediate from theorem~\ref{bernstein-no-ultra}.
\end{proof}

It remains to show that $W^P$ does not contain an $\VitEq$-transversal, and for this we use the compact balance of $P$.
The argument here will only assume that $\Gamma$ has countable colouring number; the technical assumption that some $K_{n, \omega_1}$ not embed is not needed.

\begin{lemma} \label{2-colouring-poset-compactly-balanced}
For $\Gamma$ a bipartite Borel graph of countable colouring number, $\colouringPoset{2}{\Gamma}$ is compactly balanced.
\end{lemma}

\begin{proof}
By lemma~\ref{bvc-classification-2-colouring-poset}, the balanced virtual conditions of $\colouringPoset{2}{\Gamma}$ are classified by total $2$-colourings of $\Gamma$.
These form a closed subset $C$ of the space $2^X$, where $X$ is the underlying Polish space of $\Gamma$, and hence form a compact space.
This space is definable since $\Gamma$ is Borel.
We must now verify the conditions (1), (2), and (3) from the definition of compact balance.
For this fix generic extensions $V[G] \subseteq V[H]$ by forcings $Q$ and $R$, respectively.

\begin{enumerate}
\item Let $p \in \colouringPoset{2}{\Gamma}$. By definition of the poset $p$ can be extended to a total $2$-colouring of $\Gamma$, and by lemma~\ref{bvc-classification-2-colouring-poset} this determines a balanced virtual condition
strengthening $p$.
The set $\{ f \in 2^X \mathrel{:} f \restriction \dom{p}  = p \}$ is closed in $2^X$, and hence its intersection with $C$, which is equal to $\{ f \in C : \langle \Coll{\aleph_0}{2^\nat}, \check f \rangle \le p \}$, is closed in $C$.
\item Let $\overline p \in V[G]$ be a balanced virtual condition.
Replacing $H$ by its product with a generic filter for a forcing of cardinality less than $\lambda$ if necessary, we may assume that the virtual condition $\overline p$ is realized in $V[G]$, say by $p$.
Let $f \in 2^X$ extend $p$ to a total $2$-colouring of $\Gamma$ in $V[H]$; then in $V[H]$ we have $\langle \Coll{\aleph_0}{2^\nat}, \check f \rangle \le \overline p$.
Let $\dot q \in V$ be an $R * \Coll{\aleph_0}{2^\nat}^R$-name for a total colouring extending the $R$-realization of $\overline p$, and observe that $\langle R * \Coll{\aleph_0}{2^\nat}, \dot q \rangle$ is balanced in $V$
(here ``$*$'' denotes two-step iteration of forcing).
Choose $\overline q$ a balanced virtual condition in the balance-equivalence class of $\langle R * \Coll{\aleph_0}{2^\nat}, \dot q \rangle$; clearly $\overline q \le \overline p$.
\item This is immediate by the classification of balanced virtual conditions and the fact that $\{ f \in 2^X \mathrel{:} f \restriction \dom{p} = p \}$ is closed in $2^X$ for any partial function $p \subseteq X \times 2$.
\end{enumerate}
\end{proof}

\begin{lemma} \label{parity-no-e0}
The model $W^P$ does not contain an $\VitEq$-transversal.
\end{lemma}

\begin{proof}
By theorem~\ref{compact-balance-smooth-divide},
$W^P \models \lvert 2^\nat / \VitEq \rvert > 2^{\aleph_0}$.
In any model of $\ZF$, the existence of a transversal for $\VitEq$ implies the existence of an injection from $2^\nat$ to $2^\nat / \VitEq$, which implies $2^{\aleph_0} \le \lvert 2^\nat / \VitEq \rvert$.
Hence there is no transversal for $\VitEq$ in $W^P$.
\end{proof}

\section{Independence of $\VitEq$-transversals and Nonprincipal Ultrafilters on $\nat$}

The authors of~\cite{choice-and-the-hat-game} observe that there are models in the literature witnessing that the existence of a nonprincipal ultrafilter on $\nat$ does not imply the existence of an $\VitEq$-transversal and
\emph{vice versa}, but that the constructions of which they are aware are far from elementary and require a proper class of Woodin cardinals.
They ask (Question 11) whether there are more elementary constructions witnessing this independence.
We are tempted to answer in the affirmative, but it is unclear how ``elementary'' the methods of geometric set theory are.
In any case we use balanced forcing to reduce the large cardinal assumption required to construct models witnessing the independence of nonprincipal ultrafilters on $\nat$ from $\VitEq$-transversals to a single inaccessible cardinal.

\begin{prop}
The balanced extension of $W$ by $[\nat]^\nat$, ordered by inclusion, contains a nonprincipal ultrafilter on $\nat$ but no $\VitEq$-transversal.
\end{prop}

\begin{proof}
This is corollary 9.2.5 of \cite{larson-zapletal-GST}, but we give a brief sketch.
Let $P$ be the forcing in the statement of the proposition.
As noted in \cite[sec. 7.1]{larson-zapletal-GST} (and well-known in general), forcing with $P$ adds a nonprincipal ultrafilter on $\nat$ to $W$.
This forcing is compactly balanced because its balanced virtual conditions are classified by nonprincipal ultrafilters on $\omega$, which naturally form a compact topological space (see \cite[thm. 7.1.2]{larson-zapletal-GST}).
Hence $W^P \models \lvert 2^\nat / \VitEq \rvert > 2^{\aleph_0}$, and so there is no $\VitEq$-transversal in $W^P$.
\end{proof}

\begin{prop}
The balanced extension of $W$ by the $\VitEq$-transversal poset\footnote{A natural forcing notion to add an $\VitEq$-transversal by countable approximations; see \cite[def. 6.4.4]{larson-zapletal-GST}} does not contain a
nonprincipal ultrafilter on $\nat$.
\end{prop}

\begin{proof}
The equivalence relation $\VitEq$ is trivially placid (see \cite[def. 3.3.1]{larson-zapletal-GST}), and so the $\VitEq$-transversal poset is placid\footnote{For which definition see \cite[def. 9.3.1]{larson-zapletal-GST}.}
by \cite[cor. 9.3.6]{larson-zapletal-GST}.
Placid implies Bernstein-balanced (\cite[thm. 12.2.8]{larson-zapletal-GST}), so the result follows directly from theorem~\ref{bernstein-no-ultra}.
\end{proof}

Combining these two propositions, we see that the existence of a nonprincipal ultrafilter on $\nat$ is independent of the existence of an $\VitEq$-transversal.
This partially answers Question~11 of \cite{choice-and-the-hat-game} by reducing the required large-cardinal assumption to an inaccessible cardinal.
Unpublished work of Hanul Jeon \cite{hanul}
eliminates the need of a large-cardinal assumption entirely, thus fully answering Question~11 of \cite{choice-and-the-hat-game}.
Hanul follows a different approach to choiceless independence results which utilizes recently-modernized machinery for constructing symmetric submodels of forcing extensions.

\section{Chromatic Number $2 < n < \aleph_0$}


As discussed in the preliminaries, the natural generalization of the hat game to a finite number of colours can be solved by more or less the same means as in the $2$-colour case.
Specifically, players make use of a \emph{$k$-arity function}, which can be thought of as a colouring of the graph $\prod_{n = 1}^\infty K_k$, where $K_k$ is the complete graph on $k$ vertices.

As in the case of parity functions, the full axiom of choice is not required to establish the existence of $k$-arity functions.
In particular, the existence of a nonprincipal ultrafilter on $\nat$ or the existence of an $\VitEq$-transversal each easily imply the existence of a $k$-arity function. 

We shall demonstrate in this section, under the assumption of the existence of an inaccessible cardinal, that the existence of a $k$-arity function does not imply the existence of either a nonprincipal ultrafilter on $\nat$ or an
$\VitEq$-transversal.
In fact, we shall establish that for any locally-countable Borel graph $\Gamma$ which has chromatic number $n$ in $V$, where $2 < n < \aleph_0$, there is a balanced and Bernstein-balanced forcing extension of the symmetric Solovay model
in which $\Gamma$ has chromatic number $n$, there is no nonprincipal ultrafilter on $\nat$, and there is no $\VitEq$-transversal.
We do not know how to weaken the assumption of local countability to something like countable colouring number.

When we refer to $\VitEq$ on $k^\nat$, we mean the path-connectedness relation on $\prod_{n = 0}^\infty K_k$, which is easily seen to be Borel bireducible with $\VitEq$ on $2^\nat$.

\begin{prop}
Let $\Gamma$ be a locally-countable Borel graph on a Polish space $X$ with chromatic number $n < \aleph_0$, and assume there is an inaccessible cardinal.
Then there is a model of $\ZF + \DC$ in which $\Gamma$ has chromatic number $n$, there is no nonprincipal ultrafilter on $\nat$, and there is no $\VitEq$-transversal.
\end{prop}

\begin{proof}
We use the forcing $P$ of \cite[Example 6.2.11]{larson-zapletal-GST},
which is the forcing based on the simplicial complex $\mathcal K$ of finite partial $n$-colourings of $\Gamma$ which can be extended to total colourings.
The complex $\mathcal K$ is fragmented by \cite[Theorem 6.2.7]{larson-zapletal-GST} (it is clearly locally countable; see \cite[Example 6.2.11]{larson-zapletal-GST}), and so by \cite[Example 9.3.4]{larson-zapletal-GST} 
and \cite[Therem 12.2.8]{larson-zapletal-GST} it is Bernstein-balanced. 
Consequently $W^P$ contains no nonprincipal ultrafilter on $\nat$.
Example 9.2.17 in \cite{larson-zapletal-GST} notes that $P$ is compactly balanced, which implies there is no $\VitEq$-transversal in $W^P$.
\end{proof}

Note that the forcing consisting of countable partial $n$-colourings of $\Gamma$ which can be extended to total colourings would also have sufficed in the above proof, by similar arguments as in the $2$-colouring case.
However, the assumption of local countability is still difficult to relax because there is no apparent Borel way to determine whether a countable partial colouring can be extended to a total colouring without this assumption.
Also, since the Hamming graph is locally-countable, if our goal had only been to prove theorem~\ref{main-thm} we could have used the simplicial complex graph colouring forcing employed here rather than the $2$-colouring poset
$P_2(H)$.
The value of defining $P_2(\Gamma)$ for bipartite Borel graphs $\Gamma$ is that it applies to the broader class of bipartite graphs with countable colouring number.

\section{Open Problems}

We collect here a number of questions relating to the logical strength of graph-colouring principles in $\ZF + \DC$.
The first question implicitly considers the extent to which the assumptions of countable colouring number and non-embedding of some $K_{n, \omega_1}$ in our main theorem~(\ref{main-thm}) can be relaxed.

\begin{quest}
Is there a Borel bipartite graph $\Gamma$ such that every model of $\ZF + \DC$ containing a $2$-colouring of $\Gamma$ also contains a nonprincipal ultrafilter on $\nat$?
What about the analogous question for an $\VitEq$-transversal?
\end{quest}

In the case of finite chromatic number greater than $2$ we needed to make the even stronger requirement that graphs be locally countable. Can this be relaxed?

\begin{quest}
Is there a Borel graph $\Gamma$ with chromatic number $k < \aleph_0$ and countable colouring number such that every model of $\ZF + \DC$ containing a $2$-colouring of $\Gamma$ also contains a nonprincipal ultrafilter on $\nat$?
What about the analogous question for a $\VitEq$-transversal?
\end{quest}

The following question is also of interest:

\begin{quest}
Is there a Borel graph $\Gamma$ of finite chromatic number $k$ and uncountable colouring number which has a $k$-colouring in some model of $\ZF + \DC$ which contains neither a nonprincipal ultrafilter on $\nat$ nor a
$\VitEq$-transversal?
\end{quest}

One can also weaken the above questions to ask about colouring a finite-chromatic Borel graph with some finite number of colours which may be greater than the chromatic number.
Another means of generalization is to weaken the definability requirement that graphs be Borel to requirements such as that the graph be analytic, coanalytic, or projective;
and another asks about models where there is no $E$-transversal for some nonsmooth analytic equivalence relation $E$ on a Polish space.
All of our main questions relate to the following, which we expect to be a long-term endeavor:

\begin{quest}
What are the classes of projectively-definable graphs $\Gamma$ of countable chromatic number such that every model of $\ZF + \DC$ containing a $\chi(\Gamma)$-colouring satisfies each of the following principles?
\begin{itemize}
\item There is a nonprincipal ultrafilter on $\nat$.
\item There is an $\VitEq$-transversal.
\item There is neither a nonprincipal ultrafilter on $\nat$ nor an $\VitEq$-transversal.
\end{itemize}
\end{quest}

\emph{Acknowledgement.} The author gratefully acknowledges the extensive help provided by his advisor, Justin Moore, in conducting this research and preparing this paper.
The author's research was partly supported by NSF grants DMS-1854367 and DMS-2153975.
He thanks Paul Larson for suggesting the hat game problem, for conversations about technical issues encountered along the way, and for multiple careful reviews of this paper.
Hanul Jeon also took a great interest in this work and conversations with him have resulted in the resolution of large-cardinal issues which the author had initially considered future work.

\bibliographystyle{plain}
\bibliography{math}

\end{document}